\documentclass[a4paper]{article}

\usepackage{amsmath}
\usepackage{amsfonts}
\usepackage{amssymb}
\usepackage{amsthm}

\usepackage{bibentry}
\nobibliography*

\newtheorem{theorem}{Theorem}
\newtheorem{lemma}{Lemma}

\theoremstyle{definition}
\newtheorem{example}{Example}

\begin{document}
\bibliographystyle{abbrv}

\title{Symmetric Nonnegative $ 5 \times 5 $ Matrices Realizing Previously Unknown Region}
\author{Oren Spector}
\date{\today}
\maketitle

\begin{abstract}
In this paper we present some symmetric nonnegative $ 5 \times 5 $ matrix families that realize a previously unknown region. We also prove that these and other symmetric nonnegative $ 5 \times 5 $ matrix families are closed under perturbations first presented in~\bibentry{RefWorks:21}.
\end{abstract}

\section{Introduction}

Let $ \sigma = \left( \lambda_1, \lambda_2, \dots, \lambda_n \right) $ be a list of complex numbers. The problem of determining necessary and sufficient conditions for $ \sigma $ to be the eigenvalues of a nonnegative $ n \times n $ matrix is called the nonnegative inverse eigenvalue problem (NIEP). When $ \sigma $ is a list of real numbers the problem of determining necessary and sufficient conditions for $ \sigma $ to be the eigenvalues of a nonnegative $ n \times n $ matrix is called the real nonnegative inverse eigenvalue problem (RNIEP). When $ \sigma $ is a list of real numbers the problem of determining necessary and sufficient conditions for $ \sigma $ to be the eigenvalues of a symmetric nonnegative $ n \times n $ matrix is called the symmetric nonnegative inverse eigenvalue problem (SNIEP). All three problems are currently unsolved for $ n \geq 5 $.

Loewy and London~\cite{RefWorks:39} have solved NIEP for $ n = 3 $ and RNIEP for $ n = 4 $. Moreover, RNIEP and SNIEP are the same for $ n \leq 4 $. This can be seen from papers by Fiedler~\cite{RefWorks:59} and Loewy and London~\cite{RefWorks:39}. However, it has been shown by Johnson et al.~\cite{RefWorks:47} that RNIEP and SNIEP are different in general. More results about the general NIEP, RNIEP and SNIEP can be found in~\cite{RefWorks:46}. Other results about SNIEP for $ n = 5 $ can be found in ~\cite{RefWorks:45,RefWorks:36,RefWorks:16,RefWorks:79}.

Guo~\cite{RefWorks:21} proved the realizability of some real eigenvalue perturbations of a realizable spectrum in NIEP and RNIEP. Laffey~\cite{RefWorks:66} extended this result to some non-real eigenvalue perturbations. Rojo and Soto~\cite{RefWorks:82} extended this result to symmetric nonnegative circulant matrices.

This paper is organized as follows: In Section~\ref{sec:preliminaries} we present some notations and results concerning SNIEP. In Section~\ref{sec:patterns} we present some matrix patterns and their properties. In Section~\ref{sec:unknown_region} we show that these matrix patterns realize a previously unknown region. In Section~\ref{sec:perturbations} we prove Guo-type results for several families of $ 5 \times 5 $ symmetric nonnegative matrices, including those presented in Section~\ref{sec:patterns}. Finally, in Section~\ref{sec:examples} we show some exaples that illustrate our results.

\section{Preliminaries and Notations} \label{sec:preliminaries}

By the Perron-Frobenius Theorem we know that if $ \sigma $ is the spectrum of a nonnegative matrix $ A $ then
\begin{equation}
\rho = \max_{\lambda \in \sigma} \lvert \lambda \rvert \in \sigma. \label{PF condition}
\end{equation}
We call $ \rho $ the Perron eigenvalue of $ A $. Also, by the nonnegativity of $ A $ we have
\begin{equation}
\sum_{\lambda \in \sigma} \lambda \geq 0. \label{trace condition}
\end{equation}

The following result is due to Fiedler \cite{RefWorks:59}, who extended a result due to Suleimanova \cite{RefWorks:44}:
\begin{theorem} \label{th:Suleimanova}
Let $ \lambda_1 \geq 0 \geq \lambda_2 \geq \dots \geq \lambda_n $ and $ \sum_{i=1}^{n} \lambda_i \geq 0 $. Then there exists a symmetric nonnegative $ n \times n $ matrix with a spectrum $ \sigma = \left( \lambda_1, \lambda_2, \dots, \lambda_n \right) $, where $ \lambda_1 $ is its Perron eigenvalue.
\end{theorem}

By~\cite{RefWorks:39} and \cite{RefWorks:59} we have the following theorem:
\begin{theorem} \label{th:SNIEP_n=4}
Let $ \sigma = \left( \lambda_1, \lambda_2, \lambda_3, \lambda_4 \right) $ be a list of real numbers. Necessary and sufficient conditions for $ \sigma $ to be the spectrum of a $ 4 \times 4 $ symmetric nonnegative matrix are \eqref{PF condition} and \eqref{trace condition}.
\end{theorem}

The following necessary condition due to McDonald and Neumann \cite{RefWorks:45} was extended by Loewy and McDonald \cite{RefWorks:36}:
\begin{lemma} \label{lem:MN condition}
Let $ A $ be a $ 5 \times 5 $ nonnegative symmetric matrix with a spectrum $ \sigma = \left( \lambda_1, \lambda_2, \dots, \lambda_5 \right) $ such that $ \lambda_1 \geq \lambda_2 \geq \lambda_3 \geq \lambda_4 \geq \lambda_5 $. Then
\begin{equation}
\lambda_1 + \lambda_3 + \lambda_4 \geq 0. \label{MN condition}
\end{equation}
\end{lemma}

\begin{lemma} \label{lem:SNIEP_n=5_MN}
Let a list of real numbers $ \sigma = \left( \lambda_1, \lambda_2, \dots, \lambda_5 \right) $ satisfy $ \lambda_1 \geq \lambda_2 > 0 \geq \lambda_3 \geq \lambda_4 \geq \lambda_5 $. Necessary and sufficient conditions for $ \sigma $ to be the spectrum of a $ 5 \times 5 $ symmetric nonnegative matrix are \eqref{PF condition}, \eqref{trace condition} and \eqref{MN condition}.
\end{lemma}

\begin{proof}
Let $ A $ be a $ 5 \times 5 $ symmetric nonnegative matrix with a spectrum $ \sigma = \left( \lambda_1, \lambda_2, \dots, \lambda_5 \right) $ such that $ \lambda_1 \geq \lambda_2 \geq \lambda_3 \geq \lambda_4 \geq \lambda_5 $. Therefore, \eqref{PF condition} and \eqref{trace condition} are met. By Lemma~\ref{lem:MN condition} condition \eqref{MN condition} is met as well.

Let a list of real numbers $ \sigma = \left( \lambda_1, \lambda_2, \dots, \lambda_5 \right) $ satisfy $ \lambda_1 \geq \lambda_2 > 0 \geq \lambda_3 \geq \lambda_4 \geq \lambda_5 $, \eqref{PF condition}, \eqref{trace condition} and \eqref{MN condition}. Realizing $ 5 \times 5 $ symmetric nonnegative matrices for this case are given in~\cite{RefWorks:16}.
\end{proof}

We denote the spectrum of a $ 5 \times 5 $ nonnegative real symmetric matrix by the list of real numbers $ \sigma = \left( \lambda_1, \lambda_2, \dots, \lambda_5 \right) $. Unless otherwise stated we assume that
\begin{equation}
\lambda_1 \geq \lambda_2 \geq \dots \geq \lambda_5 \geq -\lambda_1. \label{monotonicity condition}
\end{equation}
Let $ \alpha \in \mathbb{R} $. We define $ \alpha \cdot \sigma = \left( \alpha \lambda_1, \alpha \lambda_2, \dots, \alpha \lambda_5 \right) $.

By~\cite{RefWorks:36} we know that the only unknown realizable region of SNIEP for $ n = 5 $ is when $ \lambda_{3} > \sum_{i=1}^{5} \lambda_i $.

\section{Matrix Patterns} \label{sec:patterns}

In this section we define several matrix patterns and calculate their spectra. Let $ \sigma = \left( \lambda_1, \lambda_2, \dots, \lambda_5 \right) $ be a list of complex numbers and let $ g \in \mathbb{C} $. Define

\[
A \left( \sigma \right) = \frac{1}{2 u}
\left( \begin {array}{ccccc} 2 u e_{{1}}&0&\sqrt {2 u^3}&0&\sqrt {2 u^3}\\\noalign{\medskip}0&0&0&2 w&2 \sqrt {v}
\\\noalign{\medskip}\sqrt {2 u^3}&0&0&2 \sqrt {v}&2 r
\\\noalign{\medskip}0&2 w&2 \sqrt {v}&0&0\\\noalign{\medskip}\sqrt {2 u^3}&2 \sqrt {v}&2 r&0&0\end {array} \right)
,\]

\[
B \left( \sigma, g \right) =
 \left( \begin {array}{ccccc} g&\sqrt {l}&\sqrt {m}&0&0
\\\noalign{\medskip}\sqrt {l}&0&0&0&k\\\noalign{\medskip}\sqrt {m}&0&e
_{{1}}-2 g&\sqrt {m}&0\\\noalign{\medskip}0&0&\sqrt {m}&g&\sqrt {l}
\\\noalign{\medskip}0&k&0&\sqrt {l}&0\end {array} \right)
,\]

where
\begin{alignat*}{2}
e_n &= e_n( \sigma ) &&= \sum_{1 \leq j_1 < j_2 < \dots < j_n \leq 5}{\,\prod_{i=1}^{n}{\lambda_{j_i}}} \quad (n = 1, 2 \dots, 5), \\
u &= u( \sigma ) &&= -e_2( \sigma ) -{\lambda_{{2}}}^{2}-{\lambda_{{5}}}^{2}, \\
v &= v( \sigma ) &&= - \left( \lambda_{{3}} + \lambda_{{5}} \right)  \left( \lambda_{{4}} + \lambda_{{5}} \right)  \left( \lambda_{{2}} + \lambda_{{4}} \right) \left( \lambda_{{2}} + \lambda_{{3}} \right)  \left( \lambda_{{1}} + \lambda_{{2}} \right)  \left( \lambda_{{1}} + \lambda_{{5}} \right), \\
w &= w( \sigma ) &&= \lambda_{{2}} \lambda_{{5}} e_1( \sigma ) -\lambda_{{1}}\lambda_{{3}}\lambda_{{4}}, \\
r &= r( \sigma ) &&= e_3( \sigma ) + e_1( \sigma ) \left( {\lambda_{{2}}}^{2} + {\lambda_{{5}}}^{2} \right), \\
k &= k( \sigma, g ) &&= g - \lambda_{{3}} - \lambda_{{5}}, \\
l &= l( \sigma, g ) &&= \left( g - \lambda_{{3}} \right)  \left( \lambda_{{5}} - g \right), \\
m &= m( \sigma, g ) &&= - {g}^{2} + e_{{1}}( \sigma ) g - \frac{1}{2} \left( e_{{2}}( \sigma ) + {\lambda_{{3}}}^{2} + {\lambda_{{5}}}^{2} \right).
\end{alignat*}

In order for $ A( \sigma ) $ to be well-defined we assume that $ u \neq 0 $ and that if $ z = r \exp( \imath \varphi ) \in \mathbb{C} $, where $ r $ is a nonnegative real number and $ -\pi < \varphi \leq \pi $, then $ \sqrt{z} = \sqrt{r} \exp( \imath \frac{\varphi}{2} ) $. 

\begin{lemma} \label{lem:spectrum_pattern1}
Assume $ \sigma = \left( \lambda_1, \lambda_2, \dots, \lambda_5 \right) $ is a list of complex numbers and $ u( \sigma ) \neq 0 $. Then the spectrum of $ A( \sigma ) $ is $ \sigma $.
\end{lemma}

\begin{proof}
The characteristic polynomial of $ A( \sigma ) $ is \[ P_A( z ) = z^5 + q_4 z^4 + q_3 z^3 + q_2 z^2 + q_1 z + q_0 ,\] where
\begin{align*}
q_4 &= -e_1, \\
q_3 &= -{\frac {{u}^{3}+2 v + {r}^{2} + {w}^{2}}{{u}^{2}}}, \\
q_2 &= {\frac {-r{u}^{2}+2 e_{{1}}v+e_{{1}}{r}^{2}+e_{{1}}{w}^{2}}{{u}^{2}}}, \\
q_1 &= {\frac {-2 vwr + {w}^{2}{u}^{3} + {u}^{3}v + {v}^{2} + {r}^{2}{w}^{2}}{{u}^{4}}}, \\
q_0 &= -{\frac {-2 e_{{1}}vwr + {u}^{2}vw-{u}^{2}{w}^{2}r+e_{{1}}{r}^{2}{w}^{2}+e_{{1}}{v}^{2}}{{u}^{4}}}.
\end{align*}

The following expressions are polynomials in $ \lambda_1, \dots, \lambda_5 $. Therefore, it is straightforward to check that
\begin{align*}
-u^2 q_3 &= {u}^{3}+2 v + {r}^{2} + {w}^{2} = -u^2 e_2, \\
u^2 q_2 &= -r{u}^{2}+2 e_{{1}}v+e_{{1}}{r}^{2}+e_{{1}}{w}^{2} = - u^2 e_3, \\
u^4 q_1 &= -2 vwr + {w}^{2}{u}^{3} + {u}^{3}v + {v}^{2} + {r}^{2}{w}^{2} = u^4 e_4, \\
-u^4 q_0 &= -2 e_{{1}}vwr + {u}^{2}vw - {u}^{2}{w}^{2}r + e_{{1}}{r}^{2}{w}^{2} + e_{{1}}{v}^{2} = u^4 e_5.
\end{align*}

As $ u \neq 0 $ we get that
\begin{align*}
P_A( z ) &= z^5 - e_1 z^4 + e_2 z^3 - e_3 z^2 + e_4 z - e_5 \\
 &= ( z - \lambda_1 ) ( z - \lambda_2 ) ( z - \lambda_3 ) ( z - \lambda_4 ) ( z - \lambda_5 )
\end{align*}
and so the spectrum of $ A( \sigma ) $ is $ \sigma $.
\end{proof}

Let
\[ Q_\sigma( z ) = 2 {z}^{3} - 2 \left( \lambda_{{3}} + \lambda_{{5}} \right) {z}^{2} - \left( e_{{2}} + \left( \lambda_{{3}} - \lambda_{{5}} \right) ^{2} \right) z + e_{{3}}  + e_{{1}} \left( {\lambda_{{3}}}^{2} + {\lambda_{{5}}}^{2} \right) .
\]

\begin{lemma} \label{lem:spectrum_pattern2}
Assume $ \sigma = \left( \lambda_1, \lambda_2, \dots, \lambda_5 \right) $ is a list of complex numbers and let $ g $ be a root of the polynomial $ Q_\sigma( z ) $. Then the spectrum of $ B( \sigma, g ) $ is $ \sigma $.
\end{lemma}

\begin{proof}
The characteristic polynomial of $ B( \sigma, g ) $ is \[ P_B( z ) = z^5 + q_4 z^4 + q_3 z^3 + q_2 z^2 + q_1 z + q_0 ,\] where
\begin{align*}
q_4 &= -e_1, \\
q_3 &= -{k}^{2} - 2 l - 2 m + 2 ge_{{1}} - 3 {g}^{2} = e_2, \\
q_2 &= -2 gl + e_{{1}}{k}^{2} + 2 le_{{1}} + 2 mg-{g}^{2}e_{{1}} + 2 {g}^{3} = Q_\sigma( g ) - e_3, \\
q_1 &= 4 {g}^{2}l + 2 {k}^{2}m - 2 {k}^{2}ge_{{1}} + 3 {k}^{2}{g}^{2} + 2 ml - 2 gle_{{1}}+{l}^{2} \\
 &= - \left( \lambda_{{3}} + \lambda_{{5}} \right) Q_\sigma( g ) + e_4, \\
q_0 &= -2 lkm + 2 g{l}^{2} - 2 {k}^{2}mg + {k}^{2}{g}^{2}e_{{1}} - 2 {k}^{2}{g}^{
3} - {l}^{2}e_{{1}} \\
 &= \lambda_{{3}} \lambda_{{5}} Q_\sigma( g ) - e_5.
\end{align*}

As we assumed $ Q_\sigma( g ) = 0 $, we get that $ q_2 = -e_3 $, $ q_1 = e_4 $ and $ q_0 = -e_5 $. Therefore,
\begin{align*}
P_B( z ) &= z^5 - e_1 z^4 + e_2 z^3 - e_3 z^2 + e_4 z - e_5 \\
 &= ( z - \lambda_1 ) ( z - \lambda_2 ) ( z - \lambda_3 ) ( z - \lambda_4 ) ( z - \lambda_5 )
\end{align*}
and so the spectrum of $ B( \sigma, g ) $ is $ \sigma $.
\end{proof}

\section{Unknown Realizable Region} \label{sec:unknown_region}

In this section we show what part of the unknown region the new patterns realize.

\begin{lemma} \label{lem:nonnegative_pattern1}
Let $ \sigma = \left( \lambda_1, \lambda_2, \dots, \lambda_5 \right) $ be a list of real numbers. Assume \eqref{trace condition} and
\begin{align}
& \lambda_1 \geq \lambda_2 \geq \dots \geq \lambda_5 > - \lambda_1, \label{assumption 1} \\
& \lambda_{3} > e_1( \sigma ), \label{assumption 3} \\
& r( \sigma ) \geq 0. \label{assumption 4}
\end{align}
Then $ A( \sigma ) $ is nonnegative, and so $ \sigma $ is realizable by a nonnegative symmetric matrix.
\end{lemma}

\begin{proof}
By~\eqref{trace condition} element $ (1, 1) $ of $ A( \sigma ) $ is nonnegative. It suffices to show that $ u( \sigma ) > 0 $, $ v( \sigma ) \geq 0 $ and $ w( \sigma ) \geq 0 $ in order to show that $ A( \sigma ) $ is nonnegative.

We first note that $ v( \sigma ) > 0 $, because by~\eqref{assumption 1}, \eqref{trace condition} and \eqref{assumption 3} we have $ \lambda_3 > 0 > \lambda_4 $ and $ \lambda_2 + \lambda_4 < 0 $. Therefore, the first three terms of $ v( \sigma ) $ are negative and the last three terms are positive. Together with the minus sign of $ v( \sigma ) $ we get a positive value.

By~\eqref{assumption 1}, \eqref{trace condition} and \eqref{assumption 3} $ \lambda_1 > 0 $. Note that \[ u( \sigma ) = {\lambda_1}^2 u \left( \frac{1}{\lambda_1} \cdot \sigma \right), \quad w( \sigma ) = {\lambda_1}^3 w \left( \frac{1}{\lambda_1}  \cdot \sigma \right) .\] Therefore, it suffices to show that $ u( \sigma ) > 0, w( \sigma ) \geq 0 $ under the assumption that $ \lambda_1 = 1 $.

We define a different parametrization of the spectrum $ \sigma $:
\[ x = \lambda_2, \quad y = \lambda_3, \quad d = \lambda_2 + \lambda_3 + \lambda_4, \quad t = e_1( \sigma ) .\] Then, $ \sigma = ( 1, x, y, d-x-y, -d-1+t ) $ and $ 1 \geq x \geq y \geq d-x-y \geq -d-1+t > -1 $. By~\eqref{assumption 1} $ d < t $ and by~\eqref{assumption 3} $ y > t $, so $ y \in \left( t, \min( x, -x+2d+1-t ) \right] $.

For given $ x, d, t $ we look at $ u, w $ as functions of $ y \in \mathbb{R} $. Then we have
\begin{align*}
u( y ) &= {y}^{2}+ \left( x-d \right) y-xd-d+dt+t-{t}^{2}, \\
w( y ) &= {y}^{2}+ \left( x-d \right) y-xdt-tx+{t}^{2}x.
\end{align*}

Obviously, these functions have an extremum at $ y = \frac{1}{2} ( d - x ) $. By~\eqref{assumption 1} and \eqref{assumption 3} we get $ -x+d-t < 0 $, so $ u( y ) $ and $ w( y ) $ have a minimum at this point. Moreover, as $ u( y ) $ and $ w( y ) $ are quadratic polynomials in $ y $, they are monotonically increasing when $ y > \frac{1}{2} ( d - x ) $.

We have $ x \geq y > t > d $, so $ \frac{1}{2} ( d - x ) < 0 $. By~\eqref{trace condition} $ t \geq 0 $, so $ t > \frac{1}{2} ( d - x ) $. Therefore, if we can show that for $ y = t $, $ u( y ) > 0 $ and $ w( y ) \geq 0 $, then it will be true for $ y > t $ as well.

In case $ y = t $ we get
\begin{align*}
u( t ) &= - \left( 1+x \right)  \left( d-t \right), \\
w( t ) &= -t \left( 1+x \right)  \left( d-t \right).
\end{align*}
As $ t \geq 0 $, $ t > d $ and $ x > 0 $, indeed $ u( t ) > 0 $ and $ w( t ) \geq 0 $. Therefore, $ A( \sigma ) $ is nonnegative. By Lemma~\ref{lem:spectrum_pattern1} the spectrum of the symmetric matrix $ A( \sigma ) $ is $ \sigma $, so $ \sigma $ is realizable by a nonnegative symmetric matrix.
\end{proof}

The condition $ \lambda_5 > -\lambda_1 $ in Lemma~\ref{lem:nonnegative_pattern1} is crucial, as the following lemma shows:
\begin{lemma} \label{lem:unrealizable_spectrum}
Let $ \sigma = \left( \lambda_1, \lambda_2, \dots, \lambda_5 \right) $ be a list of real numbers. Assume $ \sigma $ meets conditions \eqref{trace condition}, \eqref{assumption 3} and
\begin{align}
\lambda_1 \geq \lambda_2 \geq \dots \geq \lambda_5 = -\lambda_1 . \label{assumption u1}
\end{align}
Then $ \sigma $ is not realizable by a $ 5 \times 5 $ nonnegative matrix.
\end{lemma}

\begin{proof}
By~\eqref{trace condition}, \eqref{assumption 3} and \eqref{assumption u1} we have that $ \lambda_3 > 0 $ and $ \lambda_3 + \lambda_4 \leq \lambda_2 + \lambda_4 < 0 $. Assume that $ \sigma $ is realizable by a $ 5 \times 5 $ nonnegatie matrix $ A $. By the Perron-Frobenius Theorem $ \lambda_1 + \lambda_4 \geq 0 $, and so $ \lambda_2 < \lambda_1 $. If $ A $ is irreducible then it has an index of cyclicity $ h $. As the spectrum of $ A $ contains only real values and as $ \lambda_5 = -\lambda_1 $, then by the Perron-Frobenius Theorem we must have that $ h = 2 $ and that $ \sigma $ is invariant under multiplication by $ -1 $. However, this contradicts the fact that $ \lambda_3 > 0 $. If $ A $ is reducible then $ \sigma $ can be partitioned into disjoint subsets (in terms of element names and not necessarily element values) forming the spectra of irreducible nonnegative matrices. A subset containing $ \lambda_5 $ must contain $ \lambda_1 $ by the Perron-Frobenius Theorem. Such a subset of $ \sigma $ must also be invariant under multiplication by $ -1 $. As no other element of $ \sigma $ is zero and no two other elements of $ \sigma $ are opposite to one another, such a subset cannot contain any other element of $ \sigma $. On the other hand, a subset containing $ \lambda_4 $ must contain $ \lambda_1 $ by the Perron-Frobenius Theorem. Therefore, $ \sigma $ cannot be partitioned as requried, so $ A $ cannot be reducible. We conclude that no $ 5 \times 5 $ nonnegative matrix $ A $ can realize $ \sigma $. 

\end{proof}

\begin{lemma} \label{lem:nonnegative_pattern2}
Let $ \sigma = \left( \lambda_1, \lambda_2, \dots, \lambda_5 \right) $ be a list of real numbers. Assume $ \sigma $ meets conditions \eqref{trace condition}, \eqref{assumption 3} and \eqref{monotonicity condition}. Further assume that $ g $ is a real root of the polynomial $ Q_\sigma( z ) $ and
\begin{equation}
0 \leq g \leq \frac{1}{2} e_1( \sigma ). \label{assumption 6}
\end{equation}
Then $ B( \sigma, g ) $ is nonnegative, and so $ \sigma $ is realizable by a nonnegative symmetric matrix.
\end{lemma}

\begin{proof}
By~\eqref{assumption 6} the main diagonal of $ B( \sigma, g ) $ is nonnegative. By~\eqref{monotonicity condition}, \eqref{trace condition} and \eqref{assumption 3} we have $ \lambda_3 > e_1( \sigma ) \geq \frac{1}{2} e_1( \sigma ) \geq g \geq 0 > \lambda_5 $ and $ \lambda_3 + \lambda_5 < 0 \leq g $. Therefore, $ l( \sigma, g ) > 0 $ and $ k( \sigma, g ) > 0 $. Note that $ m( \sigma, z ) $ can be regarded as a quadratic polynomial in $ z $. Its discriminant is
\begin{align*}
\Delta &= {e_1( \sigma )}^2 - 2 \left( e_2( \sigma ) + {\lambda_3}^{2} + {\lambda_5}^{2} \right) \\
 &= \left( {\lambda_4}^2 - {\lambda_3}^2 \right) + {\lambda_2}^2 + \left( {\lambda_1}^2 - {\lambda_5}^2 \right) .
\end{align*}
By~\eqref{monotonicity condition}, \eqref{trace condition} and \eqref{assumption 3} we have $ -\lambda_1 \leq \lambda_5 \leq \lambda_4 < -\lambda_3 < 0 $. Therefore, $ \Delta \geq {\lambda_2}^2 > 0 $ and so $ m( \sigma, z ) $ has two real roots:
\[ z_0^{\pm} = \frac{1}{2} e_1( \sigma ) \pm \frac{1}{2} \sqrt{ \Delta } .\]
As the coefficient of $ z^2 $ in $ m( \sigma, z ) $ is negative, $ m( \sigma, z ) \geq 0 $ for any $ z \in [ z_0^-, z_0^+ ] $. Also, we know that
\begin{align*}
z_0^- & \leq \frac{1}{2} e_1( \sigma ) - \frac{1}{2} \lambda_2 \\
 &= \frac{1}{2} \left( \lambda_1 + \lambda_3 + \lambda_4 + \lambda_5 \right) \\
 & \leq \frac{1}{2} \left( \lambda_1 + \lambda_2 + \lambda_4 + \lambda_5 \right) < 0 \\
 & \leq g \leq \frac{1}{2} e_1( \sigma ) \leq z_0^+ .
\end{align*}
Therefore, $ m( \sigma, g ) \geq 0 $, and so $ B( \sigma, g ) $ is nonnegative. By Lemma~\ref{lem:spectrum_pattern2} the spectrum of the symmetric matrix $ B( \sigma, g ) $ is $ \sigma $, so $ \sigma $ is realizable by a nonnegative symmetric matrix.
\end{proof}

\section{Guo Perturbations of Symmetric Nonnegative Matrices} \label{sec:perturbations}

Let the list of complex numbers $ \sigma = \left( \lambda_1, \lambda_2, \dots, \lambda_n \right) $ be the spectrum of an $ n \times n $ nonnegative real matrix with $ \lambda_1 $ its Perron eigenvalue. Let $ s $ be a positive real number and let $ i \in \{ 2, 3, \dots, n \} $. We define the following lists of values:
\begin{itemize}
\item $ \sigma_i^{+s} $ is the list $ \sigma $ with $ \lambda_1 $ replaced by $ \lambda_1 + s $ and $ \lambda_i $ replaced by $ \lambda_i + s $,
\item $ \sigma_i^{-s} $ is the list $ \sigma $ with $ \lambda_1 $ replaced by $ \lambda_1 + s $ and $ \lambda_i $ replaced by $ \lambda_i - s $.
\end{itemize}

In \cite{RefWorks:21} Guo proved that if a list of complex numbers $ \sigma = \left( \lambda_1, \lambda_2, \dots, \lambda_n \right) $ is the spectrum of an $ n \times n $ nonnegative real matrix with $ \lambda_1 $ its Perron eigenvalue and $ \lambda_i \in \mathbb{R} $, then $ \sigma_i^{+s} $ and $ \sigma_i^{-s} $ are also realizable by an $ n \times n $ nonnegative real matrix. In this section we prove similar results for several families of $ 5 \times 5 $ symmetric nonnegative matrices.

In \cite{RefWorks:79} the eigenvalues of trace zero $ 5 \times 5 $ symmetric nonnegative matrices were fully characterized. A descending list of real numbers $ \lambda_{1} \geq \lambda_{2} \geq \lambda_{3} \geq \lambda_{4} \geq \lambda_{5} \geq -\lambda_{1} $ is realizable by a trace zero $ 5 \times 5 $ symmetric nonnegative matrix if and only if the following conditions hold:

\begin{enumerate}
\item $ \sum_{i=1}^{5} \lambda_{i} = 0 $,
\item $ \sum_{i=1}^{5} \lambda_{i}^3 \geq 0 $,
\item $ \lambda_{2} + \lambda_{5} \leq 0 $.
\end{enumerate}

\begin{theorem}
Let a list of real numbers $ \sigma = \left( \lambda_1, \lambda_2, \dots, \lambda_5 \right) $, where $ \lambda_1 \geq \lambda_2 \geq \dots \geq \lambda_5 \geq -\lambda_1 $, be the spectrum of a trace zero $ 5 \times 5 $ symmetric nonnegative matrix. Then $ \sigma_i^{-s} $ is realizable by a trace zero $ 5 \times 5 $ symmetric nonnegative matrix for any positive real number $ s $ and any $ i \in \{ 2, 3, 4, 5 \} $.
\end{theorem}

\begin{proof}
The sum of elements of $ \sigma $ and of $ \sigma_i^{-s} $ is the same, so the first condition is met for $ \sigma_i^{-s} $. As $ s > 0 $ we know that $ \lambda_{2} $ and $ \lambda_{5} $ can only get smaller, so the third condition holds for $ \sigma_i^{-s} $ as well. Finally, as $ -\lambda_{1} \leq \lambda_{i} \leq \lambda_{1} $ we get that $ \lambda_{1} + \lambda_{i} \geq 0 $ and $ \lambda_{1} - \lambda_{i} \geq 0 $. Therefore, 
\[
\left( \lambda_{{1}}+s \right) ^{3}+ \left( \lambda_{{i}}-s \right) ^{3} =
3 s \left( \lambda_{{1}} + \lambda_{{i}} \right)  \left( \lambda_{{1}}+
s-\lambda_{{i}} \right) + {\lambda_{{1}}}^{3} + {\lambda_{{i}}}^{3} \geq
{\lambda_{{1}}}^{3} + {\lambda_{{i}}}^{3}.
\]
We deduce that the second condition is also met for $ \sigma_i^{-s} $. Therefore, $ \sigma_i^{-s} $ is realizable by a trace zero $ 5 \times 5 $ symmetric nonnegative matrix.
\end{proof}

\begin{theorem} \label{th:Guo_known_region}
Let a list of real numbers $ \sigma = \left( \lambda_1, \lambda_2, \dots, \lambda_5 \right) $, where $ \lambda_1 \geq \lambda_2 \geq \dots \geq \lambda_5 \geq -\lambda_1 $ and $ \lambda_3 \leq \sum_{i=1}^{5} \lambda_i $, be the spectrum of a $ 5 \times 5 $ symmetric nonnegative matrix. Then $ \sigma_i^{+s} $ and $ \sigma_i^{-s} $ are realizable by a $ 5 \times 5 $ symmetric nonnegative matrix for any positive real number $ s $ and any $ i \in \{ 2, 3, 4, 5 \} $.
\end{theorem}

\begin{proof}
Let $ \sigma = \left( \lambda_1, \lambda_2, \dots, \lambda_5 \right) $ be the spectrum of a $ 5 \times 5 $ symmetric nonnegative matrix and meet conditions \eqref{monotonicity condition} and
\begin{align}
\lambda_3 \leq e_1( \sigma ). \label{assumption k3}
\end{align}

By Lemma~\ref{lem:MN condition} condition \eqref{MN condition} is met. By~\eqref{monotonicity condition} we have $ \max_{\lambda \in \sigma_i^{+s}} \lvert \lambda \rvert = \lambda_1 + s \in \sigma_i^{+s} $ and $ \max_{\lambda \in \sigma_i^{-s}} \lvert \lambda \rvert = \lambda_1 + s \in \sigma_i^{-s} $ for any $ i \in \{ 2, 3, 4, 5 \} $. We know that \eqref{trace condition} is met for $ \sigma $, so we have $ e_1( \sigma_i^{+s} ) \geq e_1( \sigma ) \geq 0 $ and $ e_1( \sigma_i^{-s} ) = e_1( \sigma ) \geq 0 $ for any $ i \in \{ 2, 3, 4, 5 \} $. Therefore, conditions \eqref{PF condition} and \eqref{trace condition} are met by $ \sigma_i^{+s} $ and $ \sigma_i^{-s} $. We look into the following cases:

\begin{enumerate}

\item $ \lambda_3 > 0 $.

\begin{enumerate}

\item $ i \in \{ 2, 4, 5 \} $. \label{case:Guo_known_1}

In this case let $ \sigma' $ be the perturbed $ \sigma $ ($ \sigma_i^{+s} $ or $ \sigma_i^{-s} $) without the element $ \lambda_3 $. By~\eqref{assumption k3} $ e_1( \sigma' ) \geq \lambda_1 + \lambda_2 + \lambda_4 + \lambda_5 \geq 0 $. Also, $ \max_{\lambda \in \sigma'} \lvert \lambda \rvert = \lambda_1 + s \in \sigma' $. Therefore, $ \sigma' $ meets conditions \eqref{PF condition} and \eqref{trace condition}, so by Theorem~\ref{th:SNIEP_n=4} there exists a $ 4 \times 4 $ symmetric nonnegative matrix $ A_1 $ realizing $ \sigma' $. Hence, $ \sigma $ is realizable by the $ 5 \times 5 $ symmetric nonnegative matrix $ A = A_1 \oplus \left( \lambda_3 \right) $.

\item $ i = 3 $.

\begin{enumerate}

\item $ \sigma_3^{+s} $.

This is similar to case~\ref{case:Guo_known_1} with $ \sigma' = \left( \lambda_1 + s, \lambda_2, \lambda_4, \lambda_5 \right) $ and $ A = A_1 \oplus \left( \lambda_3 + s \right) $.

\item $ \sigma_3^{-s} $ and $ \lambda_3 - s \geq 0 $.

This is similar to case~\ref{case:Guo_known_1} with $ \sigma' = \left( \lambda_1 + s, \lambda_2, \lambda_4, \lambda_5 \right) $ and $ A = A_1 \oplus \left( \lambda_3 - s \right) $.

\item $ \sigma_3^{-s} $ and $ \lambda_3 - s < 0 $ and $ \lambda_4 \leq 0 $. \label{case:Guo_known_2}

In this case $ \sigma_3^{-s} $ has exactly two positive elements: $ \lambda_1 + s $ and $ \lambda_2 $. If $ \lambda_3 - s \geq \lambda_5 $ then the sum of the first, third and fourth largest elements in $ \sigma_3^{-s} $ is $ ( \lambda_1 + s ) + ( \lambda_3 - s ) + \lambda_4 = \lambda_1 + \lambda_3 + \lambda_4 $. If $ \lambda_3 - s < \lambda_5 $ then the sum of the first, third and fourth largest elements in $ \sigma_3^{-s} $ is $ ( \lambda_1 + s ) + \lambda_4 + \lambda_5 > ( \lambda_1 + s ) + \lambda_4 + ( \lambda_3 - s ) = \lambda_1 + \lambda_3 + \lambda_4 $. In both cases by \eqref{MN condition} this sum is nonnegative, so $ \sigma_3^{-s} $ meets condition \eqref{MN condition}. Therefore, by Lemma~\ref{lem:SNIEP_n=5_MN}, $ \sigma_3^{-s} $ is realizable by a $ 5 \times 5 $ symmetric nonnegative matrix.

\item $ \sigma_3^{-s} $ and $ \lambda_3 - s < 0 $ and $ \lambda_4 > 0 $.

This is similar to case~\ref{case:Guo_known_1} with $ \sigma' = \left( \lambda_1 + s, \lambda_2, \lambda_3 - s, \lambda_5 \right) $ and $ A = A_1 \oplus \left( \lambda_4 \right) $.

\end{enumerate}

\end{enumerate}

\item $ \lambda_3 \leq 0 $.

\begin{enumerate}

\item $ i = 2 $.

\begin{enumerate}

\item $ \sigma_2^{+s} $ and $ \lambda_2 + s > 0 $. \label{case:Guo_known_3}

In this case $ \sigma_2^{+s} $ has exactly two positive elements: $ \lambda_1 + s $ and $ \lambda_2 + s $. By \eqref{MN condition} the sum of the first, third and fourth largest elements in $ \sigma_2^{+s} $ is $ ( \lambda_1 + s ) + \lambda_3 + \lambda_4 > \lambda_1 + \lambda_3 + \lambda_4 \geq 0 $, so $ \sigma_2^{+s} $ meets condition \eqref{MN condition}. Hence, by Lemma~\ref{lem:SNIEP_n=5_MN}, $ \sigma_2^{+s} $ is realizable by a $ 5 \times 5 $ symmetric nonnegative matrix.

\item $ \sigma_2^{+s} $ and $ \lambda_2 + s \leq 0 $. \label{case:Guo_known_4}

In this case $ \sigma_2^{+s} $ is realizable by Theorem~\ref{th:Suleimanova}, as it meets the conditions of this theorem.

\item $ \sigma_2^{-s} $ and $ \lambda_2 - s > 0 $.

This is similar to case~\ref{case:Guo_known_3}.

\item $ \sigma_2^{-s} $ and $ \lambda_2 - s \leq 0 $.

This is similar to case~\ref{case:Guo_known_4}.

\end{enumerate}

\item $ i \in \{ 3, 4 \} $.

\begin{enumerate}

\item $ \sigma_i^{+s} $ and $ \lambda_2 \leq 0 $ and $ \lambda_i + s \leq 0 $.

This is similar to case~\ref{case:Guo_known_4}.

\item $ \sigma_i^{+s} $ and $ \lambda_2 > 0 $ and $ \lambda_i + s \leq 0 $.

This is similar to case~\ref{case:Guo_known_3}.

\item $ \sigma_i^{+s} $ and $ \lambda_i + s > 0 $.

This is similar to case~\ref{case:Guo_known_1} with $ \sigma' = \left( \lambda_1 + s, \lambda_2, \lambda_{7-i}, \lambda_5 \right) $ and $ A = A_1 \oplus \left( \lambda_i + s \right) $.

\item $ \sigma_i^{-s} $ and $ \lambda_2 > 0 $.

This is similar to case~\ref{case:Guo_known_2}.

\item $ \sigma_i^{-s} $ and $ \lambda_2 \leq 0 $.

This is similar to case~\ref{case:Guo_known_4}.

\end{enumerate}

\item $ i = 5 $.

\begin{enumerate}

\item $ \sigma_5^{+s} $ and $ \lambda_2 \leq 0 $ and $ \lambda_5 + s \leq 0 $.

This is similar to case~\ref{case:Guo_known_4}.

\item $ \sigma_5^{+s} $ and $ \lambda_2 > 0 $ and $ \lambda_5 + s \leq 0 $.

This is similar to case~\ref{case:Guo_known_2}.

\item $ \sigma_5^{+s} $ and $ \lambda_5 + s > 0 $ and $ \lambda_2 \geq \lambda_5 + s $.

This is similar to case~\ref{case:Guo_known_1} with $ \sigma' = \left( \lambda_1 + s, \lambda_2, \lambda_3, \lambda_4 \right) $ and $ A = A_1 \oplus \left( \lambda_5 + s \right) $.

\item $ \sigma_5^{+s} $ and $ \lambda_5 + s > 0 $ and $ \lambda_2 < \lambda_5 + s $.

This is similar to case~\ref{case:Guo_known_1} with $ \sigma' = \left( \lambda_1 + s, \lambda_5 + s, \lambda_3, \lambda_4 \right) $ and $ A = A_1 \oplus \left( \lambda_2 \right) $. 

\item $ \sigma_5^{-s} $ and $ \lambda_2 > 0 $.

This is similar to case~\ref{case:Guo_known_3}.

\item $ \sigma_5^{-s} $ and $ \lambda_2 \leq 0 $.

This is similar to case~\ref{case:Guo_known_4}.

\end{enumerate}

\end{enumerate}

\end{enumerate}

\end{proof}

\begin{theorem} \label{th:small_Guo_pattern1}
Let $ \sigma = \left( \lambda_1, \lambda_2, \dots, \lambda_5 \right) $ be a list of real numbers. Assume $ \sigma $ meets the conditions of Lemma~\ref{lem:nonnegative_pattern1}. Then $ \sigma_i^{-s} $ meets the conditions of Lemma~\ref{lem:nonnegative_pattern1} for sufficiently small positive real number $ s $ and any $ i \in \{ 2, 3, 4, 5 \} $, and so is realizable by a $ 5 \times 5 $ symmetric nonnegative matrix.
\end{theorem}

\begin{proof}
As we wish to prove the result for sufficiently small values of $ s $, we may assume that if $ \lambda_i > \lambda_{i+1} $ then $ \lambda_i - s \geq \lambda_{i+1} $.

Note that $ e_1( \sigma_i^{-s} ) = e_1( \sigma ) \geq 0 $ for any positive real number $ s $ and any $ i \in \{ 2, 3, 4, 5 \} $, so \eqref{trace condition} is valid for $ \sigma_i^{-s} $.

We divide the proof into several cases:
\begin{enumerate}

\item $ i = 2 $.

\begin{enumerate}

\item $ \lambda_2 > \lambda_3 $. \label{case:pattern1_lambda_2>lambda_3}

We have $ \lambda_1 + s > \lambda_2 - s \geq \lambda_3 $ and $ \lambda_5 > -\lambda_1 > -( \lambda_1 + s ) $, so \eqref{assumption 1} is valid for $ \sigma_2^{-s} $.

Since $ \lambda_3 $ is the third element of $ \sigma_2^{-s} $ and $ \lambda_3 > e_1( \sigma ) = e_1( \sigma_2^{-s} ) $ then \eqref{assumption 3} is valid for $ \sigma_2^{-s} $.

If $ r( \sigma ) > 0 $ then, by continuity of $ r( \sigma ) $, we have $ r( \sigma_2^{-s} ) > 0 $ for sufficiently small $ s $. This means \eqref{assumption 4} is valid for $ \sigma_2^{-s} $. Assume $ r( \sigma ) = 0 $. Then \[ r( \sigma_2^{-s} ) = r( \sigma_2^{-s} ) - r( \sigma ) = s \left( \lambda_{{1}} + \lambda_{{2}} \right) \left( -2 \lambda_{{2}} - \lambda_{{3}}-\lambda_{{4}}-\lambda_{{5}}+s \right) . \] As $ s > 0 $ and $ \lambda_1 + \lambda_2 > 0 $ it suffices to prove that $ 2 \lambda_{{2}} + \lambda_{{3}} + \lambda_{{4}} + \lambda_{{5}} \leq 0 $ in order to have $ r( \sigma_2^{-s} ) > 0 $. This will show \eqref{assumption 4} is valid for $ \sigma_2^{-s} $.

Collecting terms we get
\begin{align*}
r( \sigma ) &= r_4( \sigma ) \lambda_{{4}} + \lambda_{{1}}\lambda_{{2}}\lambda_{{3}} + \lambda_{{1}}\lambda_{{3}}\lambda_{{5}} + \lambda_{{2}}\lambda_{{3}}\lambda_{{5}} + \lambda_{{1}}\lambda_{{2}}\lambda_{{5}} \\
&+ \left( \lambda_{{1}} + \lambda_{{2}} + \lambda_{{3}} + \lambda_{{5}} \right)  \left( {\lambda_{{2}}}^{2} + {\lambda_{{5}}}^{2} \right),
\end{align*}
where \[ r_4( \sigma ) = \left( \lambda_{{2}} + \lambda_{{3}} + \lambda_{{5}} \right) \lambda_{{1}} + \lambda_{{3}}\lambda_{{5}} + \lambda_{{2}}\lambda_{{3}} + \lambda_{{2}}\lambda_{{5}} + {\lambda_{{5}}}^{2} + {\lambda_{{2}}}^{2} . \]

Assume $ r_4( \sigma ) = 0 $. If $ \lambda_2 + \lambda_3 + \lambda_5 = 0 $ then
\begin{align} \label{eq:lambda_4_coeff}
\lambda_{{3}}\lambda_{{5}} + \lambda_{{2}}\lambda_{{3}} + \lambda_{{2}}\lambda_{{5}} + {\lambda_{{5}}}^{2} + {\lambda_{{2}}}^{2} = 0 .
\end{align}
Substituting $ \lambda_5 = -\lambda_2 - \lambda_3 $ in \eqref{eq:lambda_4_coeff} we get $ \lambda_{{2}} \left( \lambda_{{2}} + \lambda_{{3}} \right) = 0 $, which is impossible because by~\eqref{assumption 3}, $ \lambda_2 \geq \lambda_3 > 0 $. Therefore, $ \lambda_2 + \lambda_3 + \lambda_5 \neq 0 $ and so,
\begin{align} \label{eq:lambda_1}
\lambda_1 = -{\frac {\lambda_{{3}}\lambda_{{5}} + \lambda_{{2}}\lambda_{{3}} + \lambda_{{2}}\lambda_{{5}} + {\lambda_{{5}}}^{2} + {\lambda_{{2}}}^{2}}{\lambda_{{2}} + \lambda_{{3}} + \lambda_{{5}}}} .
\end{align}
Substituting \eqref{eq:lambda_1} in $ u( \sigma ) $ gives \[ u( \sigma ) = -{\frac {\lambda_{{4}} \left( \lambda_{{3}} + \lambda_{{5}} \right) \left( \lambda_{{2}} + \lambda_{{3}} \right) }{\lambda_{{2}} + \lambda_{{3}} + \lambda_{{5}}}} . \] By~\eqref{assumption 3}, $ \lambda_4 < 0 $, $ \lambda_2 + \lambda_3 > 0 $ and $ \lambda_3 + \lambda_5 < 0 $. By Lemma~\ref{lem:nonnegative_pattern1}, $ u( \sigma ) > 0 $. Therefore, we must have $ \lambda_2 + \lambda_3 + \lambda_5 < 0 $.

Assume $ r_4( \sigma ) \neq 0 $. Then, as we assumed $ r( \sigma ) = 0 $,
\begin{align} \label{eq:lambda_4}
\lambda_4 &= -\frac{\lambda_{{1}}\lambda_{{2}}\lambda_{{3}} + \lambda_{{1}}\lambda_{{3}}\lambda_{{5}} + \lambda_{{2}}\lambda_{{3}}\lambda_{{5}} + \lambda_{{1}}\lambda_{{2}}\lambda_{{5}}}{r_4( \sigma )} \\
&-\frac{\left( \lambda_{{1}} + \lambda_{{2}} + \lambda_{{3}} + \lambda_{{5}} \right)  \left( {\lambda_{{2}}}^{2} + {\lambda_{{5}}}^{2} \right)}{r_4( \sigma )} . \nonumber
\end{align}
Substituting \eqref{eq:lambda_4} in $ u( \sigma ) $ gives
\begin{align} \label{eq:u_sub_lambda_4}
u( \sigma ) = -\frac{ \left( \lambda_{{3}} + \lambda_{{5}} \right)  \left( \lambda_{{2}} + \lambda_{{3}} \right)  \left( \lambda_{{1}} + \lambda_{{5}} \right)  \left( \lambda_{{1}} + \lambda_{{2}} \right) }{r_4( \sigma )} .
\end{align}
By~\eqref{assumption 1} and \eqref{assumption 3}, the left term in the numerator of \eqref{eq:u_sub_lambda_4} is negative and the rest of the terms in the numerator are positive. By Lemma~\ref{lem:nonnegative_pattern1}, $ u( \sigma ) > 0 $. Therefore, we must have that $ r_4( \sigma ) > 0 $. Also, by~\eqref{assumption 3}, $ \lambda_1 + \lambda_2 + \lambda_4 + \lambda_5 < 0 $. Substituting \eqref{eq:lambda_4} in this expression we get
\begin{align} \label{eq:e_1_minus_lambda_3_sub_lambda_4}
\frac{ \left( \lambda_{{1}} + \lambda_{{5}} \right)  \left( \lambda_{{1}} + \lambda_{{2}} \right)  \left( \lambda_{{2}} + \lambda_{{3}} + \lambda_{{5}} \right) }{r_4( \sigma )} < 0 .
\end{align}
As we already know that $ r_4( \sigma ) > 0 $ and that the first two terms of the numerator of \eqref{eq:e_1_minus_lambda_3_sub_lambda_4} are positive, then we must have $ \lambda_2 + \lambda_3 + \lambda_5 < 0 $.

Regardless of the value of $ r_4( \sigma ) $ we have shown that $ \lambda_2 + \lambda_3 + \lambda_5 < 0 $. Therefore, $ 2 \lambda_2 + \lambda_3 + \lambda_4 + \lambda_5 < \lambda_2 + \lambda_4 < 0 $.

\item $ \lambda_2 = \lambda_3 > \lambda_4 $.

This case is covered by case~\ref{case:pattern1_lambda_3>lambda_4}.

\item $ \lambda_2 = \lambda_3 = \lambda_4 > \lambda_5 $.

This case is covered by case~\ref{case:pattern1_lambda_4>lambda_5}.

\item $ \lambda_2 = \lambda_3 = \lambda_4 = \lambda_5 $.

This case is covered by case~\ref{case:pattern1_lambda_5}.

\end{enumerate}

\item $ i = 3 $.

\begin{enumerate}

\item $ \lambda_3 > \lambda_4 $. \label{case:pattern1_lambda_3>lambda_4}

We have $ \lambda_1 + s > \lambda_2 > \lambda_3 - s \geq \lambda_4 $ and $ \lambda_5 > -\lambda_1 > -( \lambda_1 + s ) $, so \eqref{assumption 1} is valid for $ \sigma_3^{-s} $.

Since $ \lambda_3 - s $ is the third element of $ \sigma_3^{-s} $ and $ \lambda_3 > e_1( \sigma ) = e_1( \sigma_3^{-s} ) $ we require that $ s < \lambda_3 - e_1( \sigma ) $, so that \eqref{assumption 3} is valid for $ \sigma_3^{-s} $.

We also have \[ r( \sigma_3^{-s} ) - r( \sigma ) = -s \left( \lambda_{{2}} + \lambda_{{4}} + \lambda_{{5}} \right)  \left( s + \lambda_{{1}} - \lambda_{{3}} \right) . \] We know $ s > 0 $. By~\eqref{assumption 1} $ \lambda_1 \geq \lambda_3 $ and by~\eqref{assumption 3} $ \lambda_{{2}} + \lambda_{{4}} + \lambda_{{5}} < 0 $, so $ r( \sigma_3^{-s} ) - r( \sigma ) > 0 $. This shows \eqref{assumption 4} is valid for $ \sigma_3^{-s} $.

\item $ \lambda_3 = \lambda_4 > \lambda_5 $.

This case is covered by case~\ref{case:pattern1_lambda_4>lambda_5}.

\item $ \lambda_3 = \lambda_4 = \lambda_5 $.

This case is covered by case~\ref{case:pattern1_lambda_5}.

\end{enumerate}

\item $ i = 4 $.

\begin{enumerate}

\item $ \lambda_4 > \lambda_5 $. \label{case:pattern1_lambda_4>lambda_5}

We have $ \lambda_1 + s > \lambda_2 $, $ \lambda_4 - s \geq \lambda_5 $ and $ \lambda_5 > -\lambda_1 > -( \lambda_1 + s ) $, so \eqref{assumption 1} is valid for $ \sigma_4^{-s} $.

Since $ \lambda_3 $ is the third element of $ \sigma_4^{-s} $ and $ \lambda_3 > e_1( \sigma ) = e_1( \sigma_4^{-s} ) $ then \eqref{assumption 3} is valid for $ \sigma_4^{-s} $.

If $ r( \sigma ) > 0 $ then, by continuity of $ r( \sigma ) $, we have $ r( \sigma_4^{-s} ) > 0 $ for sufficiently small $ s $. This means \eqref{assumption 4} is valid for $ \sigma_4^{-s} $. Assume $ r( \sigma ) = 0 $. Then \[ r( \sigma_4^{-s} ) = r( \sigma_4^{-s} ) - r( \sigma ) = -s \left( \lambda_{{2}} + \lambda_{{3}} + \lambda_{{5}} \right)  \left( s + \lambda_{{1}} - \lambda_{{4}} \right) . \] We know $ s > 0 $ and by~\eqref{assumption 1} $ s + \lambda_1 - \lambda_4 > 0 $. In case~\ref{case:pattern1_lambda_2>lambda_3} we already showed that when $ r( \sigma ) = 0 $ we have $ \lambda_{{2}} + \lambda_{{3}} + \lambda_{{5}} < 0 $. Therefore, $ r( \sigma_4^{-s} ) > 0 $. This shows \eqref{assumption 4} is valid for $ \sigma_4^{-s} $.

\item $ \lambda_4 = \lambda_5 $.

This case is covered by case~\ref{case:pattern1_lambda_5}.

\end{enumerate}

\item $ i = 5 $. \label{case:pattern1_lambda_5}

We have $ \lambda_1 + s > \lambda_2 $, and by~\eqref{assumption 1} $ \lambda_5 - s > -( \lambda_1 + s ) $, so \eqref{assumption 1} is valid for $ \sigma_5^{-s} $.

Since $ \lambda_3 $ is the third element of $ \sigma_5^{-s} $ and $ \lambda_3 > e_1( \sigma ) = e_1( \sigma_5^{-s} ) $ then \eqref{assumption 3} is valid for $ \sigma_5^{-s} $.

We also have \[ r( \sigma_5^{-s} ) - r( \sigma ) = s \left( \lambda_{{1}} + \lambda_{{5}} \right)  \left( s - \lambda_{{2}} - \lambda_{{3}} - \lambda_{{4}} - 2 \lambda_{{5}} \right) . \] We know $ s > 0 $. By~\eqref{assumption 1} $ \lambda_1 + \lambda_5 > 0 $ and by~\eqref{assumption 3} $ \lambda_{{2}} + \lambda_{{3}} + \lambda_{{4}} + 2 \lambda_{{5}} < 0 $, so $ r( \sigma_5^{-s} ) - r( \sigma ) > 0 $. This shows \eqref{assumption 4} is valid for $ \sigma_5^{-s} $.

\end{enumerate}
\end{proof}

\begin{theorem} \label{th:small_Guo_pattern2}
Let $ \sigma = \left( \lambda_1, \lambda_2, \dots, \lambda_5 \right) $ be a list of real numbers and $ g \in \mathbb{R} $. Assume $ \sigma $ and $ g $ meet the conditions of Lemma~\ref{lem:nonnegative_pattern2} and also the condition
\begin{equation}
r( \sigma ) < 0 . \label{assumption 7}
\end{equation}
Then, for sufficiently small positive real number $ s $ and $ i \in \{ 2, 3, 4, 5 \} $, there exists $ \tilde{g} \in \mathbb{R} $ (depending on $ s $ and $ i $) such that $ \sigma_i^{-s} $ and $ \tilde{g} $ meet the conditions of Lemma~\ref{lem:nonnegative_pattern2} and \eqref{assumption 7}, and so $ \sigma_i^{-s} $ is realizable by a $ 5 \times 5 $ symmetric nonnegative matrix.
\end{theorem}

\begin{proof}
As we wish to prove the result for sufficiently small values of $ s $, we may assume that if $ \lambda_i > \lambda_{i+1} $ then $ \lambda_i - s \geq \lambda_{i+1} $.

Note that $ e_1( \sigma_i^{-s} ) = e_1( \sigma ) \geq 0 $ for any positive real number $ s $ and any $ i \in \{ 2, 3, 4, 5 \} $, so \eqref{trace condition} is valid for $ \sigma_i^{-s} $.

By our assumption $ r ( \sigma ) < 0 $. Also, by continuity of $ r ( \sigma ) $ in $ \lambda_1 $ and $ \lambda_i $, for sufficiently small positive real number $ s $ we have $ r( \sigma_i^{-s} ) < 0 $. Therefore, \eqref{assumption 7} is valid for $ \sigma_i^{-s} $.

By our assumptions there is $ g \in \left[ 0, \frac{1}{2} e_1( \sigma ) \right] $ such that $ Q_\sigma( g ) = 0 $ and $ B( \sigma, g ) $ is nonnegative. By~\eqref{monotonicity condition}, \eqref{trace condition} and \eqref{assumption 3} we know that $ \lambda_2 \geq \lambda_3 > 0 $. Therefore, by~\eqref{assumption 7} we have \[ e_3( \sigma ) + e_1( \sigma ) \left( {\lambda_3}^2 + {\lambda_5}^2 \right) \leq e_3( \sigma ) + e_1( \sigma ) \left( {\lambda_2}^2 + {\lambda_5}^2 \right) = r ( \sigma ) < 0 , \] which means that the $ Q_\sigma( 0 ) < 0 $. By continuity of polynomials, for sufficiently small positive real number $ s $ and $ i \in \{ 2, 3, 4, 5 \} $, $ Q_{\sigma_i^{-s}}( 0 ) < 0 $.

The inflection point of $ Q_\sigma( z ) $ is at $ \frac{1}{3} \left( \lambda_3 + \lambda_5 \right) $. By~\eqref{monotonicity condition}, \eqref{trace condition} and \eqref{assumption 3}, $ \lambda_3 + \lambda_5 < 0 $ and as the coefficient of $ z^3 $ in $ Q_\sigma( z ) $ is positive, then for $ z \geq 0 $ the polynomial $ Q_\sigma( z ) $ is convex. This means that there can only be a single real root of $ Q_\sigma( z ) $ in the range $ \left( 0, \frac{1}{2} e_1( \sigma ) \right] $.

If $ g < \frac{1}{2} e_1( \sigma ) $ then, by continuity of the roots of a polynomial, for sufficiently small positive real number $ s $ and $ i \in \{ 2, 3, 4, 5 \} $, $ Q_{\sigma_i^{-s}}( z ) $ has a real root $ \tilde{g} \in \left( 0, \frac{1}{2} e_1( \sigma ) \right) $ so \eqref{assumption 6} is valid for $ Q_{\sigma_i^{-s}}( z ) $. We will therefore assume that
\begin{equation}
Q_\sigma \left( \frac{1}{2} e_1( \sigma ) \right) = 0 . \label{assumption 8}
\end{equation}

We divide the proof into several cases:
\begin{enumerate}

\item $ i = 2 $.

\begin{enumerate}

\item $ \lambda_2 > \lambda_3 $. \label{case:pattern2_lambda_2>lambda_3}

We have $ \lambda_1 + s > \lambda_2 - s \geq \lambda_3 $ and $ \lambda_5 \geq -\lambda_1 > -( \lambda_1 + s ) $, so \eqref{monotonicity condition} is valid for $ \sigma_2^{-s} $.

Since $ \lambda_3 $ is the third element of $ \sigma_2^{-s} $ and $ \lambda_3 > e_1( \sigma ) = e_1( \sigma_2^{-s} ) $ then \eqref{assumption 3} is valid for $ \sigma_2^{-s} $.

We have \[ Q_{\sigma_2^{-s}}( z ) - Q_\sigma( z ) = s \left( s + \lambda_{{1}} - \lambda_{{2}} \right)  \left( z - \lambda_{{3}} - \lambda_{{4}} - \lambda_{{5}} \right) . \] By~\eqref{monotonicity condition}, \eqref{trace condition} and \eqref{assumption 3}, $ \lambda_1 \geq \lambda_2 $ and $ \lambda_3 + \lambda_4 + \lambda_5 < 0 $. Together with $ s > 0 $ we get that the above expression is positive for $ z \geq 0 $, and in particular for $ z = \frac{1}{2} e_1( \sigma ) $, so that by~\eqref{assumption 8} \[ Q_{\sigma_2^{-s}} \left( \frac{1}{2} e_1( \sigma ) \right) = Q_{\sigma_2^{-s}} \left( \frac{1}{2} e_1( \sigma ) \right) - Q_\sigma \left( \frac{1}{2} e_1( \sigma ) \right) > 0 . \] Thus, $ Q_{\sigma_2^{-s}}( z ) $ must have a real root $ \tilde{g} \in \left( 0, \frac{1}{2} e_1( \sigma ) \right) $, so \eqref{assumption 6} is valid for $ Q_{\sigma_2^{-s}}( z ) $.

\item $ \lambda_2 = \lambda_3 > \lambda_4 $.

This case is covered by case~\ref{case:pattern2_lambda_3>lambda_4}.

\item $ \lambda_2 = \lambda_3 = \lambda_4 > \lambda_5 $.

This case is covered by case~\ref{case:pattern2_lambda_4>lambda_5}.

\item $ \lambda_2 = \lambda_3 = \lambda_4 = \lambda_5 $.

This case is covered by case~\ref{case:pattern2_lambda_5}.

\end{enumerate}

\item $ i = 3 $.

\begin{enumerate}

\item $ \lambda_3 > \lambda_4 $. \label{case:pattern2_lambda_3>lambda_4}

We have $ \lambda_1 + s > \lambda_2 > \lambda_3 - s \geq \lambda_4 $ and $ \lambda_5 \geq -\lambda_1 > -( \lambda_1 + s ) $, so \eqref{monotonicity condition} is valid for $ \sigma_3^{-s} $.

Since $ \lambda_3 - s $ is the third element of $ \sigma_3^{-s} $ and $ \lambda_3 > e_1( \sigma ) = e_1( \sigma_3^{-s} ) $ we require that $ s < \lambda_3 - e_1( \sigma ) $, so that \eqref{assumption 3} is valid for $ \sigma_3^{-s} $.

We have
\begin{align*}
Q_{\sigma_3^{-s}} \left( \frac{1}{2} e_1( \sigma ) \right) &- Q_\sigma \left( \frac{1}{2} e_1( \sigma ) \right) = \left( \lambda_{{1}} + \lambda_{{3}} \right) s^2 \\
 &+ \frac{1}{2} s \left( \left( \lambda_{{2}} - \lambda_{{3}} + \lambda_{{4}} - \lambda_{{5}} \right) \left( \lambda_{{1}} + \lambda_{{3}} \right) + \left( \lambda_{{2}} + \lambda_{{4}} \right) ^{2} \right) \\
 &+ \frac{1}{2} s \left( \left( \lambda_{{1}} - \lambda_{{5}} \right) \left( \lambda_{{1}} + \lambda_{{5}} \right) + {\lambda_{{1}}}^{2} + \lambda_{{3}} \left( \lambda_{{1}} - \lambda_{{3}} \right)  \right) .
\end{align*}
By~\eqref{monotonicity condition} we have $ \lambda_1 \geq \lambda_2 \geq \lambda_3 \geq \lambda_4 \geq \lambda_5 \geq - \lambda_1 $, and by~\eqref{trace condition} and \eqref{assumption 3} $ \lambda_1 \geq \lambda_3 > 0 $. Together with $ s > 0 $ we get that the above expression is positive. Hence, by~\eqref{assumption 8} \[ Q_{\sigma_3^{-s}} \left( \frac{1}{2} e_1( \sigma ) \right) = Q_{\sigma_3^{-s}} \left( \frac{1}{2} e_1( \sigma ) \right) - Q_\sigma \left( \frac{1}{2} e_1( \sigma ) \right) > 0 , \] and so $ Q_{\sigma_3^{-s}}( z ) $ must have a real root $ \tilde{g} \in \left( 0, \frac{1}{2} e_1( \sigma ) \right) $. Therefore, \eqref{assumption 6} is valid for $ Q_{\sigma_3^{-s}}( z ) $.

\item $ \lambda_3 = \lambda_4 > \lambda_5 $.

This case is covered by case~\ref{case:pattern2_lambda_4>lambda_5}.

\item $ \lambda_3 = \lambda_4 = \lambda_5 $.

This case is covered by case~\ref{case:pattern2_lambda_5}.

\end{enumerate}

\item $ i = 4 $.

\begin{enumerate}

\item $ \lambda_4 > \lambda_5 $. \label{case:pattern2_lambda_4>lambda_5}

We have $ \lambda_1 + s > \lambda_2 $, $ \lambda_4 - s \geq \lambda_5 $ and $ \lambda_5 \geq -\lambda_1 > -( \lambda_1 + s ) $, so \eqref{monotonicity condition} is valid for $ \sigma_4^{-s} $.

Since $ \lambda_3 $ is the third element of $ \sigma_4^{-s} $ and $ \lambda_3 > e_1( \sigma ) = e_1( \sigma_4^{-s} ) $ then \eqref{assumption 3} is valid for $ \sigma_4^{-s} $.

Note that
\begin{align*}
Q_\sigma & \left( \frac{1}{2} e_1( \sigma ) \right) = \frac{1}{2} \lambda_{{2}} \left( \lambda_{{2}} + \lambda_{{5}} \right) \left( \lambda_{{1}} + \lambda_{{2}} + \lambda_{{3}} + \lambda_{{4}} - \lambda_{{5}} \right) \\
&- \frac{1}{2} \left( \lambda_{{1}} - \lambda_{{2}} - \lambda_{{3}} + \lambda_{{4}} - \lambda_{{5}} \right) \left( \lambda_{{1}} + \lambda_{{3}} \right) \left( \lambda_{{3}} + \lambda_{{4}} \right) \\
&+ \frac{1}{4} \left( \lambda_{{1}} - \lambda_{{2}} - \lambda_{{3}} + \lambda_{{4}} - \lambda_{{5}} \right) \left( \lambda_{{1}} + \lambda_{{2}} + \lambda_{{3}} + \lambda_{{4}} - \lambda_{{5}} \right) e_1( \sigma ) .
\end{align*}
If $ \lambda_{{1}} - \lambda_{{2}} - \lambda_{{3}} + \lambda_{{4}} - \lambda_{{5}} \leq 0 $, then together with \eqref{monotonicity condition}, \eqref{trace condition} and \eqref{assumption 3} we get that the above expression is negative. This contradicts \eqref{assumption 8}, and so
\begin{align}
\lambda_{{1}} - \lambda_{{2}} - \lambda_{{3}} + \lambda_{{4}} - \lambda_{{5}} > 0 . \label{positive expression 1}
\end{align}

We have
\begin{align*}
Q_{\sigma_4^{-s}} \left( \frac{1}{2} e_1( \sigma ) \right) &- Q_\sigma \left( \frac{1}{2} e_1( \sigma ) \right) = \\
& \frac{1}{2} s \left( \lambda_{{1}} - \lambda_{{2}} - \lambda_{{3}} + \lambda_{{4}} - \lambda_{{5}} \right)  \left( s + \lambda_{{1}} - \lambda_{{4}} \right) .
\end{align*}
By~\eqref{monotonicity condition}, \eqref{assumption 8}, \eqref{positive expression 1} and the fact that $ s > 0 $ we have \[ Q_{\sigma_4^{-s}} \left( \frac{1}{2} e_1( \sigma ) \right) = Q_{\sigma_4^{-s}} \left( \frac{1}{2} e_1( \sigma ) \right) - Q_\sigma \left( \frac{1}{2} e_1( \sigma ) \right) > 0 . \] Thus, $ Q_{\sigma_4^{-s}}( z ) $ must have a real root $ \tilde{g} \in \left( 0, \frac{1}{2} e_1( \sigma ) \right) $, so \eqref{assumption 6} is valid for $ Q_{\sigma_4^{-s}}( z ) $.

\item $ \lambda_4 = \lambda_5 $.

This case is covered by case~\ref{case:pattern2_lambda_5}.

\end{enumerate}

\item $ i = 5 $. \label{case:pattern2_lambda_5}

We have $ \lambda_1 + s > \lambda_2 $, and by~\eqref{assumption 1} $ \lambda_5 - s \geq -( \lambda_1 + s ) $, so \eqref{monotonicity condition} is valid for $ \sigma_5^{-s} $.

Since $ \lambda_3 $ is the third element of $ \sigma_5^{-s} $ and $ \lambda_3 > e_1( \sigma ) = e_1( \sigma_5^{-s} ) $ then \eqref{assumption 3} is valid for $ \sigma_5^{-s} $.

Note that
\begin{align*}
Q_\sigma & \left( \frac{1}{2} e_1( \sigma ) \right) = \\
& \lambda_{{1}} \lambda_{{2}} \lambda_{{4}} + \frac{1}{4} \left( \lambda_{{1}} + \lambda_{{2}} - \lambda_{{3}} + \lambda_{{4}} - \lambda_{{5}} \right)  \left( {\lambda_{{1}}}^{2} + {\lambda_{{2}}}^{2} - {\lambda_{{3}}}^{2} + {\lambda_{{4}}}^{2} - {\lambda_{{5}}}^{2} \right) .
\end{align*}
By~\eqref{monotonicity condition}, \eqref{trace condition} and \eqref{assumption 3} $ \lambda_{{1}} + \lambda_{{2}} - \lambda_{{3}} + \lambda_{{4}} - \lambda_{{5}} > 0 $, so by~\eqref{assumption 8} we get \[ {\lambda_{{1}}}^{2} + {\lambda_{{2}}}^{2} - {\lambda_{{3}}}^{2} + {\lambda_{{4}}}^{2} - {\lambda_{{5}}}^{2} = - \frac { 4 \lambda_{{1}} \lambda_{{2}} \lambda_{{4}} } { \lambda_{{1}} + \lambda_{{2}} - \lambda_{{3}} + \lambda_{{4}} - \lambda_{{5}} } . \] Therefore, by~\eqref{monotonicity condition}, \eqref{trace condition} and \eqref{assumption 3}
\begin{align}
\lambda_{{2}} \lambda_{{4}} &+ \frac{1}{2} \left( {\lambda_{{1}}}^{2} + {\lambda_{{2}}}^{2} - {\lambda_{{3}}}^{2} + {\lambda_{{4}}}^{2} - {\lambda_{{5}}}^{2} \right) = \nonumber \\
& \qquad - \frac{ \lambda_{{2}} \lambda_{{4}} \left( \lambda_{{3}} - \lambda_{{2}} - \lambda_{{4}} + \lambda_{{1}} + \lambda_{{5}} \right) } { \lambda_{{1}} + \lambda_{{2}} - \lambda_{{3}} + \lambda_{{4}} - \lambda_{{5}} } > 0 . \label{positive expression 2}
\end{align}

We have
\begin{align*}
Q_{\sigma_5^{-s}} \left( \frac{1}{2} e_1( \sigma ) \right) &- Q_\sigma \left( \frac{1}{2} e_1( \sigma ) \right) = \left( \lambda_{{1}} + \lambda_{{5}} \right) {s}^{2} \\
&+ \frac{1}{2} s \left( \lambda_{{1}} + \lambda_{{5}} \right) \left( \lambda_{{1}} + \lambda_{{2}} - \lambda_{{3}} + \lambda_{{4}} - \lambda_{{5}} \right) \\
&+ s \left( \lambda_{{2}} \lambda_{{4}} + \frac{1}{2} \left( {\lambda_{{1}}}^{2} + {\lambda_{{2}}}^{2} - {\lambda_{{3}}}^{2} + {\lambda_{{4}}}^{2} - {\lambda_{{5}}}^{2} \right) \right) .
\end{align*}
By~\eqref{monotonicity condition}, \eqref{trace condition}, \eqref{assumption 3}, \eqref{assumption 8}, \eqref{positive expression 2} and the fact that $ s > 0 $ we have \[ Q_{\sigma_5^{-s}} \left( \frac{1}{2} e_1( \sigma ) \right) = Q_{\sigma_5^{-s}} \left( \frac{1}{2} e_1( \sigma ) \right) - Q_\sigma \left( \frac{1}{2} e_1( \sigma ) \right) > 0 . \] Thus, $ Q_{\sigma_5^{-s}}( z ) $ must have a real root $ \tilde{g} \in \left( 0, \frac{1}{2} e_1( \sigma ) \right) $, so \eqref{assumption 6} is valid for $ Q_{\sigma_5^{-s}}( z ) $.

\end{enumerate}
\end{proof}

\begin{lemma} \label{lem:spectral_radius}
Let $ A $ be a symmetric nonnegative matrix with a spectral radius $ \rho( A ) $. Then the elements of $ A $ are in the range $ [ 0, \rho( A ) ] $.
\end{lemma}

\begin{proof}
It is well-known that for an $ n \times n $ symmetric matrix $ A = ( a_{ij} ) $ we have \[ \rho( A ) = \max_{\|  x \|_2 = 1} \| Ax \|_2 . \] Let $ \mathbf{e}_i $ be the $ i^\text{th} $ unit vector in $ {\mathbb{R}}^n $. Then, as $ A $ is nonnegative, \[ 0 \leq a_{ij} = \sqrt{ {a_{ij}}^2 } \leq  \sqrt{ \sum_{k=1}^{n} {a_{kj}}^2 } = \| A \mathbf{e}_j \|_2 \leq \rho( A ) . \]
\end{proof}

In the next theorem we do not assume that the elements of $ \sigma $ are monotonically decreasing.

\begin{theorem} \label{th:Guo_extension}
Let $ i \in \{ 2, 3, \dots, n \} $ be fixed. Assume that if a list of real numbers $ \sigma = \left( \lambda_1, \lambda_2, \dots, \lambda_n \right) $ is realizable by a nonnegative symmetric matrix with Perron eigenvalue $ \lambda_1 $, then there exists a positive real number $ s_0 $ (depending on $ \sigma $ and $ i $), such that $ \sigma_i^{+s} $ ($ \sigma_i^{-s} $) is realizable by a nonnegative symmetric matrix with Perron eigenvalue $ \lambda_1 + s $ for all $ 0 \leq s < s_0 $. Then $ \sigma_i^{+s} $ ($ \sigma_i^{-s} $) is realizable by a symmetric nonnegative matrix for any positive real number $ s $.
\end{theorem}

\begin{proof}
We follow the path of Laffey in \cite{RefWorks:66}. Let $ i \in \{ 2, 3, \dots, n \} $ be fixed and let $ \sigma = \left( \lambda_1, \lambda_2, \dots, \lambda_n \right) $ be a fixed realizable spectrum of a nonnegative symmetric matrix. Assume that for some $ r > 0 $, $ \sigma_i^{+r} $ ($ \sigma_i^{-r} $) is not realizable by a symmetric nonnegative matrix. Let \[ W^{\pm} = \left\{ t > 0 \mid \sigma_i^{\pm s} \text{ is realizable for all } s \text{ with } 0 \leq s < t \right\} . \] Note that by assumption, there exists a positive real number $ s_0 $ such that $ s_0 \in W^+ $ ($ s_0 \in W^- $). Also, $ W^+ $ ($ W^- $) is bounded above by $ r $, so it has a supremum, which we denote by $ w $. Obviously, $ w \geq s_0 > 0 $. Let $ \{ w_k \} $  ($ k = 1, 2, \dots $) be a strictly increasing sequence of positive real numbers with limit $ w $. Note that $ \sigma_i^{+w_k} $ ($ \sigma_i^{-w_k} $) is realizable by a symmetric nonnegative matrix, or else $ w \leq w_k $, in contradiction of the definition of the sequence $ \{ w_k \} $. Let $ A_k $ be a symmetric nonnegative matrix having the spectrum $ \sigma_i^{+w_k} $ ($ \sigma_i^{-w_k} $). By Lemma~\ref{lem:spectral_radius} the entries of all $ A_k $ lie in the interval $ [ 0, \lambda_1 + w ] $. But then the sequence $ \{ A_k \} $ must have a convergent subsequence. Let $ A_0 $ be the limit of such a subsequence. Then, by continuity of the spectrum, $ A_0 $ has spectrum $ \sigma_i^{+w} $ ($ \sigma_i^{-w} $). Also, $ A_0 $ is a symmetric nonnegative matrix. Therefore, $ \sigma_i^{+s} $ ($ \sigma_i^{-s} $) is realizable by a symmetric nonnegative matrix for all $ 0 \leq s \leq w $. Also, by our assumption, there exists a positive real number $ s_1 $, such that $ \sigma_i^{+(w+s)} $ ($ \sigma_i^{-(w+s)} $) is realizable by a nonnegative symmetric matrix with Perron eigenvalue $ \lambda_1 + w + s $ for all $ 0 \leq s < s_1 $. We conclude that $ w + s_1 \in W^+ $ ($ w + s_1 \in W^- $), but this contradicts the fact that $ w $ is the supremum of $ W^+ $ ($ W^- $). 
\end{proof}

\begin{theorem} \label{th:Guo_pattern1}
Let $ \sigma = \left( \lambda_1, \lambda_2, \dots, \lambda_5 \right) $ be a list of real numbers. Assume $ \sigma $ meets the conditions of Lemma~\ref{lem:nonnegative_pattern1}. Then $ \sigma_i^{-s} $ is realizable by a $ 5 \times 5 $ symmetric nonnegative matrix for any positive real number $ s $ and any $ i \in \{ 2, 3, 4, 5 \} $.
\end{theorem}

\begin{proof}
By Theorem~\ref{th:small_Guo_pattern1} we know that $ \sigma_i^{-s} $ is realizable by a $ 5 \times 5 $ symmetric nonnegative matrix for sufficiently small positive real number $ s $ and any $ i \in \{ 2, 3, 4, 5 \} $. We follow the proof of  Theorem~\ref{th:Guo_extension} with the addition of the conditions of Lemma~\ref{lem:nonnegative_pattern1}, which meets the assumptions of Theorem~\ref{th:Guo_extension}. The delicate step in the proof is when looking at the limit $ A_0 $ of a subsequence $ \{ A_k \} $ of $ 5 \times 5 $ symmetric nonnegative matrices, whose spectra meet the conditions of Lemma~\ref{lem:nonnegative_pattern1}. Let $ \tilde{\sigma} = \left( \delta_1, \delta_2, \dots, \delta_5 \right) $ be the spectrum of $ A_0 $. We want to prove that $ {\tilde{\sigma}}_i^{-s} $ is realizable by a $ 5 \times 5 $ symmetric nonnegative matrix for sufficiently small positive real number $ s $ and any $ i \in \{ 2, 3, 4, 5 \} $. If this can be achieved then the rest of the proof of Theorem~\ref{th:Guo_extension} is valid as well.

$ \tilde{\sigma} $ may not meet some conditions of Lemma~\ref{lem:nonnegative_pattern1}. The problematic conditions are the strict inequalities $ \lambda_5 > -\lambda_1 $ and $ \lambda_3 > e_1( \sigma ) $, which may become an equality in the limit. We look at the different cases:
\begin{enumerate}

\item $ \delta_5 > -\delta_1 $ and $ \delta_3 > e_1( \tilde{\sigma} ) $.

In this case $ {\tilde{\sigma}}_i^{-s} $ is realizable by a $ 5 \times 5 $ symmetric nonnegative matrix for sufficiently small positive real number $ s $ and any $ i \in \{ 2, 3, 4, 5 \} $ by Theorem~\ref{th:small_Guo_pattern1}.

\item $ \delta_3 = e_1( \tilde{\sigma} ) $.

In this case $ {\tilde{\sigma}}_i^{-s} $ is realizable by a $ 5 \times 5 $ symmetric nonnegative matrix for any positive real number $ s $ and any $ i \in \{ 2, 3, 4, 5 \} $ by Theorem~\ref{th:Guo_known_region}.

\item $ \delta_5 = -\delta_1 $ and $ \delta_3 > e_1( \tilde{\sigma} ) $.

This case is ruled out by Lemma~\ref{lem:unrealizable_spectrum}, because we know that $ \tilde{\sigma} $ is realizable by the $ 5 \times 5 $ symmetric nonnegative matrix $ A_0 $.

\end{enumerate}
\end{proof}

\begin{theorem} \label{th:Guo_pattern2}
Let $ \sigma = \left( \lambda_1, \lambda_2, \dots, \lambda_5 \right) $ be a list of real numbers and $ g \in \mathbb{R} $. Assume $ \sigma $ and $ g $ meet the conditions of Theorem~\ref{th:small_Guo_pattern2}. Then $ \sigma_i^{-s} $ is realizable by a $ 5 \times 5 $ symmetric nonnegative matrix for any positive real number $ s $ and any $ i \in \{ 2, 3, 4, 5 \} $.
\end{theorem}

\begin{proof}
By Theorem~\ref{th:small_Guo_pattern2} we know that $ \sigma_i^{-s} $ is realizable by a $ 5 \times 5 $ symmetric nonnegative matrix for sufficiently small positive real number $ s $ and any $ i \in \{ 2, 3, 4, 5 \} $. We follow the proof of  Theorem~\ref{th:Guo_extension} with the addition of the conditions of Theorem~\ref{th:small_Guo_pattern2}, which meets the assumptions of Theorem~\ref{th:Guo_extension}. The delicate step in the proof is when looking at the limit $ A_0 $ of a subsequence $ \{ A_k \} $ of $ 5 \times 5 $ symmetric nonnegative matrices, whose spectra meet the conditions of Theorem~\ref{th:small_Guo_pattern2} (together with some $ g_k \in \left[ 0, \frac{1}{2} e_1( \sigma ) \right] $). Let $ \tilde{\sigma} = \left( \delta_1, \delta_2, \dots, \delta_5 \right) $ be the spectrum of $ A_0 $. We know that the sequence of polynomials $ Q_{\sigma_i^{-w_k}}( z ) $, each having a real root $ g_k \in \left[ 0, \frac{1}{2} e_1( \sigma ) \right] $, converges to the polynomial $ Q_{\tilde{\sigma}}( z ) $. By continuity of roots of polynomials, $ Q_{\tilde{\sigma}}( z ) $ must have a real root $ \tilde{g} \in \left[ 0, \frac{1}{2} e_1( \sigma ) \right] $. We want to prove that $ {\tilde{\sigma}}_i^{-s} $ is realizable by a $ 5 \times 5 $ symmetric nonnegative matrix for sufficiently small positive real number $ s $ and any $ i \in \{ 2, 3, 4, 5 \} $. If this can be achieved then the rest of the proof of Theorem~\ref{th:Guo_extension} is valid as well.

$ \tilde{\sigma} $ and $ \tilde{g} $ may not meet some conditions of Theorem~\ref{th:small_Guo_pattern2}. The problematic conditions are the strict inequalities $ \lambda_3 > e_1( \sigma ) $ and $ r( \sigma ) < 0 $, which may become an equality in the limit. We look at the different cases:
\begin{enumerate}

\item $ \delta_3 > e_1( \tilde{\sigma} ) $ and $ r( \tilde{\sigma} ) < 0 $.

In this case $ {\tilde{\sigma}}_i^{-s} $ is realizable by a $ 5 \times 5 $ symmetric nonnegative matrix for sufficiently small positive real number $ s $ and any $ i \in \{ 2, 3, 4, 5 \} $ by Theorem~\ref{th:small_Guo_pattern2}.

\item $ \delta_3 = e_1( \tilde{\sigma} ) $.

In this case $ {\tilde{\sigma}}_i^{-s} $ is realizable by a $ 5 \times 5 $ symmetric nonnegative matrix for any positive real number $ s $ and any $ i \in \{ 2, 3, 4, 5 \} $ by Theorem~\ref{th:Guo_known_region}.

\item $ \delta_3 > e_1( \tilde{\sigma} ) $ and $ r( \tilde{\sigma} ) = 0 $.

\begin{enumerate}

\item $ \delta_5 = -\delta_1 $.

This case is ruled out by Lemma~\ref{lem:unrealizable_spectrum}, because we know that $ \tilde{\sigma} $ is realizable by the $ 5 \times 5 $ symmetric nonnegative matrix $ A_0 $.

\item $ \delta_5 > -\delta_1 $.

In this case $ {\tilde{\sigma}}_i^{-s} $ is realizable by a $ 5 \times 5 $ symmetric nonnegative matrix for sufficiently small positive real number $ s $ and any $ i \in \{ 2, 3, 4, 5 \} $ by Theorem~\ref{th:small_Guo_pattern1}.

\end{enumerate}

\end{enumerate}
\end{proof}

\section{Examples} \label{sec:examples}

\begin{example}
Let $ \sigma = \left( 1000, 381, 360, -641, -750 \right) $. Then $ e_1( \sigma ) = 350 $ and $ r( \sigma ) = 306540 $, so $ \sigma $ meets the conditions of Lemma~\ref{lem:nonnegative_pattern1}. Therefore, $ \sigma $ is realizable by the pattern $ A( \sigma ) $. On the other hand, \[ Q_\sigma( z ) = 2 z^3 + 780 z^2 - 169279 z - 5139810 \] has approximately the roots $ -538.3523722, -27.19321311, 175.5455853 $, so it has no root in the range $ \left[ 0, 175 \right] $. Therefore, $ \sigma $ does not meet the conditions of Lemma~\ref{lem:nonnegative_pattern2}. Also, as far we know, no other symmetric realizability criterion is satisfied by $ \sigma $. 
\end{example}

\begin{example}
Let $ \sigma = \left( 1000, 370, 367, -637, -750 \right) $. Then $ e_1( \sigma ) = 350 $ and \[ Q_\sigma( z ) = 2  z^3 + 766 z^2 - 189010 z - 901830 \] has approximately the roots $ -552.5556695, -4.683524431, 174.2391939 $, so its largest root, $ g $, is in the range $ \left[ 0, 175 \right] $. Therefore, $ \sigma $ meets the conditions of Lemma~\ref{lem:nonnegative_pattern2} and is realizable by the pattern $ B( \sigma, g ) $. On the other hand, $ r( \sigma ) = -127980 $, so $ \sigma $ does not meet the conditions of Lemma~\ref{lem:nonnegative_pattern1}. Also, as far we know, no other symmetric realizability criterion is satisfied by $ \sigma $. 
\end{example}

\section*{Acknowledgement}

I wish to thank Raphi Loewy for his helpful comments after reading an earlier draft of this paper and for suggesting simplified proofs for Lemma~\ref{lem:unrealizable_spectrum} and Theorem~\ref{th:Guo_known_region}.

\bibliography{new_realizable_region}

\end{document}